\documentclass[11pt]{amsart}
\usepackage{amssymb}
\usepackage{pstricks,pst-plot,}
\usepackage{color}
\usepackage{cancel}
\usepackage{tikz}
\usetikzlibrary{patterns}
\usepackage{array,color}
\usepackage{pict2e}

\makeatletter
\DeclareRobustCommand{\intprod}{%
  \mathbin{\mathpalette\int@prod{(0.1,0)(0.9,0)(0.9,0.8)}}%
}
\DeclareRobustCommand{\intprodr}{%
  \mathbin{\mathpalette\int@prod{(0.1,0.8)(0.1,0)(0.9,0)}}}

\newcommand{\int@prod}[2]{%
  \begingroup
  \sbox\z@{$\m@th#1+$}%
  \setlength\unitlength{\wd\z@}%
  \begin{picture}(1,1)
  \roundcap
  \polyline#2
  \end{picture}%
  \endgroup
}
\makeatother

\newcommand{\C}{\ensuremath{\mathbb{C}}}
\newcommand{\R}{\ensuremath{\mathbb{R}}}

\makeatletter
\newcommand{\sumprime}{\if@display\sideset{}{'}\sum%
            \else\sum'\fi}
\makeatother

\begin{document}

\numberwithin{equation}{section}

\newtheorem{theorem}{Theorem}[section]
\newtheorem{proposition}[theorem]{Proposition}
\newtheorem{conjecture}[theorem]{Conjecture}
\def\theconjecture{\unskip}
\newtheorem{corollary}[theorem]{Corollary}
\newtheorem{lemma}[theorem]{Lemma}
\newtheorem{observation}[theorem]{Observation}
\newtheorem{definition}{Definition}
\numberwithin{definition}{section} 
\newtheorem{remark}{Remark}
\def\theremark{\unskip}
\newtheorem{question}{Question}
\def\thequestion{\unskip}
\newtheorem{example}{Example}
\def\theexample{\unskip}
\newtheorem{problem}{Problem}

\def\vvv{\ensuremath{\mid\!\mid\!\mid}}
\def\intprod{\mathbin{\lr54}}
\def\reals{{\mathbb R}}
\def\integers{{\mathbb Z}}
\def\N{{\mathbb N}}
\def\complex{{\mathbb C}\/}
\def\dist{\operatorname{dist}\,}
\def\spec{\operatorname{spec}\,}
\def\interior{\operatorname{int}\,}
\def\trace{\operatorname{tr}\,}
\def\cl{\operatorname{cl}\,}
\def\essspec{\operatorname{esspec}\,}
\def\range{\operatorname{\mathcal R}\,}
\def\kernel{\operatorname{\mathcal N}\,}
\def\dom{\operatorname{Dom}\,}
\def\linearspan{\operatorname{span}\,}
\def\lip{\operatorname{Lip}\,}
\def\sgn{\operatorname{sgn}\,}
\def\Z{ {\mathbb Z} }
\def\e{\varepsilon}
\def\p{\partial}
\def\rp{{ ^{-1} }}
\def\Re{\operatorname{Re\,} }
\def\Im{\operatorname{Im\,} }
\def\dbarb{\bar\partial_b}
\def\eps{\varepsilon}
\def\O{\Omega}
\def\Lip{\operatorname{Lip\,}}

\def\Hs{{\mathcal H}}
\def\E{{\mathcal E}}
\def\scriptu{{\mathcal U}}
\def\scriptr{{\mathcal R}}
\def\scripta{{\mathcal A}}
\def\scriptc{{\mathcal C}}
\def\scriptd{{\mathcal D}}
\def\scripti{{\mathcal I}}
\def\scriptk{{\mathcal K}}
\def\scripth{{\mathcal H}}
\def\scriptm{{\mathcal M}}
\def\scriptn{{\mathcal N}}
\def\scripte{{\mathcal E}}
\def\scriptt{{\mathcal T}}
\def\scriptr{{\mathcal R}}
\def\scripts{{\mathcal S}}
\def\scriptb{{\mathcal B}}
\def\scriptf{{\mathcal F}}
\def\scriptg{{\mathcal G}}
\def\scriptl{{\mathcal L}}
\def\scripto{{\mathfrak o}}
\def\scriptv{{\mathcal V}}
\def\frakg{{\mathfrak g}}
\def\frakG{{\mathfrak G}}

\def\ov{\overline}
\def\suob{\;\tikz{\draw[line width=0.5pt](0, 0)--(8pt, 0)--(8pt, 8pt);}\;}

\title{Boundary behavior of the Szeg\"o kernel}
 
 \author[JuJie Wu]{JuJie Wu}
 \author[Xu Xing]{Xu Xing}

\address [Jujie Wu] {School of Mathematics (Zhuhai), Sun Yat-Sen University, Zhuhai, Guangdong 519082, P. R.
 	China}
\email{wujj86@mail.sysu.edu.cn}
 
\address[Xu Xing] {School of Mathematical Sciences, Fudan University, Shanghai, 20043, China} 
\email{16110180009@fudan.edu.cn}

\begin{abstract}
  We give a H\"ormander-type localization principle for the Szeg\"o kernel $S_\Omega(z)$. We also show that for each boundary point $z_0$, $S_\Omega(z)\gtrsim|z-z_0|^{-\frac{1}{3}}$ holds non-tangentially  for any bounded pseudoconvex domain with smooth boundary in ${\mathbb C}^2$.
  \bigskip
  
  \noindent{{\sc Mathematics Subject Classification} (2010): 32A25, 32A35, 32A40, 32T35.}
  
  \smallskip
  
  \noindent{{\sc Keywords}: Hardy space, Szeg\"o kernel, $\bar{\partial}_b$-operator, Boundary behavior} 
\end{abstract}


\thanks{The first author was supported by National Natural Science Foundation of China, No. 11601120.}

\thanks{The second author was supported by National Natural Science Foundation of China, No. 11771089}

\date{}
\maketitle
\section{Introduction}
Boundary behavior of the Bergman kernel has been studied by many authors (cf.\ \cite{Hormander65} \cite{Fefferman}, \cite{Ohsawa81},\ etc; for more information on this matter, we refer the reader to review article of Chen and Fu \cite{ChenFu2012}).  In particular, H\"ormander proved the following useful localization lemma for Bergman kernel.
\begin{theorem}[cf. \cite{Hormander65}]
  	Let $\Omega \subset\mathbb{C}^n$ be a bounded  pseudoconvex domain. Suppose that  $z^0\in\partial\Omega$ is a local holomorphic peak point and  $\Omega' \subset \Omega$ is a domain so that $\partial\Omega\cap\partial\Omega'$ contains a  neighborhood of $z^0$. Then we have
  	\begin{equation} 
  	\frac{K_\Omega (z) }{K_{\Omega'} (z)} \rightarrow 1, \ \ \ (z \rightarrow z^0)
  	\end{equation}
where $K_\Omega$ is the Bergman kernel of $\Omega$.
  \end{theorem}
Recall that $z^0$ is a local holomorphic peak point if there exist a neighborhood $U$ of $z^0$ and a function $h$ holomorphic in $\Omega\cap U$, continuous in $\overline{\Omega\cap U}$ such that $h(z^0)=1$ and $|h(z)|<1$ for all $z\in\overline{\Omega\cap U}\setminus \{z^0\}$.
\par It is natural to ask whether the analogue holds for the Szeg\"o kernel $S_\Omega$, i.e.,
\begin{problem}\label{pr:1}
Let $\Omega$ be a bounded pseudoconvex domain with smooth boundary in $\C^n$ and $z^0\in\partial\Omega$ a local holomorphic peak point. Suppose that $\Omega' \subset \Omega$ is a smooth pseudoconvex domain so that $\partial\Omega'\cap\partial\Omega$ contains a neighborhood of $z^0$. Does it follow that
$$
\frac{S_\Omega(z)}{S_{\Omega'}(z)}\rightarrow 1 \ \  (z\rightarrow z^0)\ ?
$$
\end{problem}
For the sake of simplicity, we assume that $z^0=0$ and the outer normal at $0$ is $(0',1)$, where $z'=(z_1, z_2, \cdots, z_{n-1})$.
\par To state our result, we introduce the following. 
\begin{definition}\label{de:1}
Let $\Omega\subset\mathbb{C}^n$ be a bounded pseudoconvex domain with smooth boundary. We say that $\Omega$ satisfies  holomorphic tangential transverse property (HTTP) at $0\in\partial\Omega$ if there exist a smooth defining function $\rho$ on $\Omega$ and $r>0$, such that for all $ t\in\Delta_r:=\{t\in\mathbb{C}:|t|<r\},$ $\Omega_t:=\Omega\cap\{z_n=t\}$ is a smooth open set with defining function $\rho|_{\Omega_t}$. 
\end{definition}
We shall give a partial answer of Problem \ref{pr:1} as follows.
\begin{theorem} \label{th: localization}
  	Let $\Omega \subset \mathbb{C}^n$ be a smooth bounded  pseudoconvex domain satisfying HTTP at $0\in\partial\Omega$, which is a local holomorphic peak point. Suppose that $\Omega' \subset \Omega$ is a smooth pseudoconvex domain so that $\partial\Omega\cap\partial\Omega'$ contains a  neighborhood of $0$. Then we have
  	\begin{equation} 
  	\frac{S_\Omega (z) }{S_{\Omega'} (z)} \rightarrow 1, \ \ \ (z \rightarrow 0).
  	\end{equation}
  \end{theorem}
\begin{remark}
It is not difficult to see that the condition in Theorem \ref{th: localization} is satisfied if $0$ is a strongly pseudoconvex point. 
\end{remark}
On the other hand, it is known, as a direct consequence of the Ohsawa-Takegoshi extension theorem (cf.\ \cite{Ohsawa81},\ see also \cite{Fu94}), that the Bergman kernel $K_\Omega(z)\gtrsim\delta_{\Omega}^{-2}$ for bounded pseudoconvex domains with $C^2$ boundary, where $\delta_\Omega (z)$ denotes the Euclidean distance to the boundary $\partial \Omega$. It is natural to propose the following
\begin{problem}\label{pr:2}
Let $\Omega$ be a bounded pseudoconvex domain with smooth boundary in $\C^n$. Is $S_\Omega(z)$ an exhaustion function? Moreover, is it possible to conclude that 
$$
S_\Omega(z)\gtrsim\delta_{\Omega}(z)^{-1}\ ?
$$
\end{problem}
It was shown in \cite{ChenFu11} that $S_\Omega(z)\gtrsim \delta_\Omega(z)^{-1}|\log \delta_\Omega(z)|^{-a}$ with some $a>0$ for  $\delta-$regular domains, which include bounded pseudoconvex domains of finite type and bounded domains with a defining function plurisubharmonic in the boundary. The authors of \cite{ChenFu11} also gave an affirmative answer to Problem \ref{pr:2} for convex domains. Here we will show

     \begin{theorem} \label{th:exhaustszego}
     Let $\Omega\subset {\mathbb C}^2$ be a bounded pseudoconvex domain with smooth boundary. Then for any $z^0\in\partial\Omega$, there exists a constant $C>0$, such that  
     $$
     S_\Omega(z)\geq C|z-z^0|^{-1/3}
     $$   
as $z$ tents to $z^0$ non-tangentially.
\end{theorem}
In order to prove Theorem \ref{th: localization}, we apply the method of H\"ormander \cite{Hormander65} with the $L^2$ existence theorem of $\bar\partial$ replaced by the $L^2$ existence theorem for the $\bar{\partial}_b$ due to Shaw \cite{Shaw85}. However, since the Szeg\"o kernel does not enjoy the monotonicity with respect to defining domains, some technical difficulties arise, which explains why we need the HTTP. Theorem \ref{th:exhaustszego} is proved by solving an extension problem for the Hardy space $H^2$, with the help of Shaw's  $L^2$ existence theorem for the $\bar{\partial}_b$. Nevertheless, due to the lack of weighted $L^2$ estimates, we can only deal with a weak case for $n=2$ and the growth exponent is not optimal.

\section{Preliminaries}
\subsection{Hardy space and the Szeg\"o kernel}

Let $\Omega$ be a bounded open set in $\C^n$ with smooth boundary. Let $\Omega_\varepsilon = \{ z\in \Omega: \delta_\Omega (z) > \varepsilon \}$. Following Stein \cite{Stein72} or Krantz \cite{Krantz01}, the harmonic Hardy space $h^2(\Omega)$ is the space of harmonic function $f$ on $\Omega$ satisfying
\begin{eqnarray*}
	\| f\| ^2_{h^2} := \limsup _{\varepsilon \rightarrow 0^+} \int _{ \partial \Omega_\varepsilon} |f(z)|^2 d S < \infty.
\end{eqnarray*}
It is well-known that the non-tangential limit $f^*(\zeta)$ of $f$ exists for almost every point $\zeta$ on $\partial \Omega$. Furthermore, $f^* \in L^2(\partial \Omega)$, $\|f\|_{h^2} = \|f^*\|_{L^2(\partial \Omega)}$, and \begin{eqnarray*}
	f(z) = \int _{\zeta\in\partial \Omega} P_\Omega(z , \zeta) f^*(\zeta) d S,
\end{eqnarray*}
where $P_\Omega(z,\zeta)$ is the Poisson kernel of $\Omega$.
\par The following result is essentially implicit in \cite{Krantz01}, P.\,332--334.
\begin{theorem}\label{poisson}
Let $\Omega\subset\mathbb{C}^n$ be a bounded domain with smooth boundary. If we denote by $B(c_w,r_1)$ and $B(C_w,r_2)$ the internally and externally tangent balls at $w\in\partial\Omega$ respectively, then for $z\in\Omega\setminus B(c_w,r_1)$, 
$$
\frac{1}{4^n r_1} G_{\Omega} (z,c_w)\leq P_{\Omega}(z,w)\leq\frac{2(r_2+\rm{diam}(\Omega))}{W_{2n-1}\cdot r_2}\delta^{1-2n}(z),
$$
where $G_\Omega (z,\zeta)$ is the Green function of $\Omega$ and $W_{2n-1}$ is the area measure of the unit sphere.
\end{theorem}
The holomorphic Hardy space $H^2(\Omega)$ is the space of holomorphic functions on $\Omega$ lying in  $h^2(\Omega)$. The Szeg$\rm{\ddot{o}}$ kernel $S_\Omega(z,w)$ is the reproducing kernel of $H^2(\Omega)$, i.e., 
\begin{eqnarray*}
	f(z) = \int _{\partial \Omega} S_\Omega (z,w) f(w)dS, \ \ \ \forall \, f\in H^2(\Omega), \ \  \forall \, z \in \Omega.
\end{eqnarray*}
It follows that 
\begin{eqnarray*}
	S_{\Omega}(z):=S_{\Omega}(z,z) = \sup \{ |f(z)|^2: f \in H^2(\Omega), \| f\| _{h^2{(\Omega)}} \leq 1 \}.
\end{eqnarray*}
\begin{lemma}[cf. \cite{ChenFu11}, Lemma 2.1]\label{le:0}
	Let $\Omega_2\subset\Omega_1 \subset\mathbb{C}^n$ be smooth bounded domains. For any fixed $z^0\in\Omega_2$, we have 
	\begin{eqnarray*}
		\int_{\partial \Omega_{2}}|f|^2 d S \leq C(\Omega_1,\Omega_2,z^0) \int_{\partial \Omega_{1}}|f|^2 dS,\ \ \ \ \ \forall\, f\in h^2(\Omega_1),
	\end{eqnarray*}
where
$$C(\Omega_1,\Omega_2,z^0):=\frac{\sup_{w\in\partial\Omega_{1}}P_{\Omega_{1}}(z^0,w)}{\inf_{w\in\partial\Omega_{2}}P_{\Omega_{2}}(z^0,w)}.$$

\end{lemma}
\begin{lemma}\label{le:1}
	Let $\Omega_2\subset\Omega_1 \subset\mathbb{C}^n$ be smooth bounded domains such that $0\in\partial\Omega_1$. Set $\Omega_{k,t}:=\Omega_k\cap\{z_n=t\}$, for $k=1,2$ and $t\in\C$.  If there exists a relatively compact open set $U\subset\C$  so that $\{\Omega_{k,t}\}_{t\in U}$ are smooth families of smooth domains in $\C^{n-1}$ for $k=1,2$, then
	\begin{eqnarray*}
		\int_{z'\in\partial \Omega_{2,t}}|f(z',t)|^2 d S' \leq C \int_{z'\in\partial \Omega_{1,t}}|f(z',t)|^2 dS',\ \ \ \ \ \forall\, f\in H^2(\Omega_1), \ |t|\ll 1,
	\end{eqnarray*}
where $C$ is independent of $t$.
\end{lemma}
\begin{proof}
By Lemma \ref{le:0}, for any fixed point $z_0\in\Omega_{2,t}$, $|t|\ll 1$, we have
	\begin{eqnarray*}
		\int_{z'\in\partial \Omega_{2,t}}|f(z',t)|^2 d S' \leq C(z'_0,t) \int_{z'\in\partial \Omega_{1,t}}|f(z',t)|^2 dS',\ \ \ \ \ \forall\, f\in H^2(\Omega_1),
	\end{eqnarray*}
where 
$$C(z'_0,t):=\frac{\sup_{w'\in\partial\Omega_{1,t}}P_{\Omega_{1,t}}(z'_0,w')}{\inf_{w'\in\partial\Omega_{2,t}}P_{\Omega_{2,t}}(z'_0,w')}.$$
 Since $\{\Omega_{k,t}\}_{t\in U}$ are smooth families of smooth domains, it follows from Theorem \ref{poisson} that there exists a number $R_1>0$ such that  
\begin{eqnarray*}
		C(z'_0,t) \leq \frac{\frac{2(R_{1}+\rm{diam}(\Omega_{1,t}))}{W_{2n-1} R_{1}}\delta_{\Omega_{1,t}}^{1-2n}(z'_0)}{\inf_{w'\in\partial\Omega_{2,t}}\frac{1}{4^n r_{2}}G_{\Omega_{2,t}}(z_0',c_{t,w'})},\ \ \ |t|\ll 1,
	\end{eqnarray*}
where  $B(c_{t,w'},r_{2})$ denotes  the internally tangent ball at $w'\in\partial\Omega_{2,t}$. It is easy to see that the numerator of the RHS of the above inequality can be bounded above by a constant independent of $t$ by choosing suitable  $z'_0$. 
By the well-known stability of $G_{\Omega_{2,t}}(z_0',c_{t,w'})$ in $t$ we conclude that the  denominator can also be bounded below by a constant independent of $t$.
\end{proof}
\begin{lemma}\label{le:2}
	Let $\Omega\subset\mathbb{C}^n$ be a bounded domain with smooth boundary. Set $\Omega_{t}:=\Omega\cap\{z_n=t\}$, $t\in\C$. If there exists a relatively compact open set $U\subset\C$  so that $\{\Omega_{t}\}_{t\in U}$ is a smooth family of smooth domains in $\C^{n-1}$, then for any open set $ V\subset\subset U$ , there exists a constant $C$ such that
\begin{eqnarray*}
\int_{z'\in\Omega_t}|f(z',t)|^2 dV'\leq C\int_{z'\in\partial\Omega_{t}}|f(z',t)|^2 dS', \ \ \ \ \ \forall\, f\in H^2(\Omega),\ t\in V.
\end{eqnarray*}
\end{lemma}
\begin{proof}
Let $\Omega_t^\epsilon:=\{z'\in\Omega_t: \delta_{\Omega_t}(z')>\epsilon\}$ for $\epsilon>0$. On the one hand, it is easy to see that there exist $\epsilon_0,r_0>0$ independent of $t$ so that $\Omega_t^{\epsilon_0}$ contains all centers of internally tangent balls with radius $r_0$. For $z'\in\Omega_t^{\epsilon_0}$ and $\zeta\in\partial\Omega_t$, it follows easily from the Bochner-Martinelli formula (compare  \cite{Chen}, the proof of Theorem 1.6)  that  
\begin{eqnarray*}
|f(z',t)|^2\leq C\int_{\partial\Omega_t}|f(\cdot,t)|^2dS',
\end{eqnarray*}
where $C$ is independent of $t$.
On the other hand, for $\epsilon<\epsilon_0$ and $z'_0\in\Omega^{\epsilon_0}_t$ fixed, a similar argument as the proof of Lemma \ref{le:1} yields
\begin{eqnarray*}
\int_{z'\in\partial\Omega_t^\epsilon} |f(z',t)|^2 dS'\leq C \int_{z'\in\partial\Omega_t}|f(z',t)|^2dS'.
\end{eqnarray*}
Hence
\begin{eqnarray*}
\int_{z'\in\Omega_t}|f(z',t)|^2dV' &=& \int_{\{z'\in\Omega_t:\delta_{\Omega_t}(z)>\epsilon_0\}}|f(z',t)|^2dV'+\int_0^{\epsilon_0} \int_{z'\in\partial\Omega_t^\epsilon}|f(z',t)|^2dS'd\epsilon\\
&\leq & C\int_{z'\in\partial\Omega_t}|f(z',t)|^2dS'.
\end{eqnarray*}
\end{proof}
 In order to prove Theorem \ref{th: localization}, we also need the following lemma.
\begin{lemma}\label{le:3}
Let $\Omega \subset \mathbb{C}^n$ be a smooth bounded domain satisfying HTTP at $z^0\in\partial\Omega$, where $z_0$ is a local holomorphic peak point. Set $L_t:=\{z=(z',z_n)\in\C^n,z_n=t\}$ for $t\in\C$ and $L(R):=\bigcup_{t\in\Delta(z_n^0,R)} L_t$ for any $R>0$, where $\Delta(t,R)$ denotes the disc with center $t$ and radius $R$. Then there exist $a_1>a_2>0$ such that $\left(L(a_1)\setminus L(a_2)\right) \cap \Omega$ is non-empty and for any $f\in h^2(\Omega)$, we have
$$
\int_{(L(a_1)\setminus \overline{L(a_2)})\cap \partial\Omega}|f|^2 dS\approx\int_{t\in\Pi((L(a_1)\setminus \overline{L(a_2)})\cap \partial\Omega)}dV_t\int_{\partial( \Omega\cap L_t)}|f(z',t)|^2 dS',
$$
where $\Pi:\C^n\rightarrow\C$ denotes projection $(z',z_n)\mapsto z_n$.
\end{lemma}
\begin{proof}
For the sake of simplicity, we may assume that $z^0=0$. Put
\begin{eqnarray*}
	U=\left(L(a_1)\setminus \overline{L(a_2)}\right)\cap \Omega  \ \ \ \ \text{and} \ \ K=\left(L(a_1)\setminus \overline{L(a_2)}\right)\cap \partial\Omega 
\end{eqnarray*}
 Let $\chi_\epsilon$ be the characteristic function of $\left(L(a_1)\setminus \overline{L(a_2)}\right)\cap \{ \rho(z)>-\epsilon\}$ satisfying 
\begin{eqnarray*}
	\chi_\epsilon| _ {(L(a_1)\setminus \overline{L(a_2)})\cap \{ \rho(z)>-\epsilon\}} =1 \ \ \ \ \ \text{and} \  \ \  \chi_\epsilon| _ {((L(a_1)\setminus \overline{L(a_2)})\cap \{ \rho(z)>-\epsilon\})^c} =0
\end{eqnarray*}
 where $\rho$ is a defining function of $\Omega$ and $\rho_t=\rho | _{\Omega\cap L_t}$. It is not difficult to see that
$$
\int_{K}|f|^2\frac{dS}{|d\rho|}=\lim_{\epsilon\rightarrow 0}\frac{1}{\epsilon}\int_{U}\chi_{\epsilon}|f|^2dV. \ \ \ \ \ \  
$$
Here $dS=\frac{d\rho}{|d\rho|} \;\tikz{\draw[line width=0.5pt](0, 0)--(6pt, 0)--(6pt, 6pt);}\;   \ dV$ and $ \;\tikz{\draw[line width=0.5pt](0, 0)--(6pt, 0)--(6pt, 6pt);}\; $ denotes the contraction operator. By the Fubini theorem, we obtain
\begin{eqnarray*}
	\int_{K}|f|^2\frac{dS}{|d\rho|}&=&\lim_{\epsilon\rightarrow 0}\frac{1}{\epsilon}\int_{t\in\Pi(U)}dV_t \int_{\Omega\cap L_t}\chi_{\epsilon}(\cdot,t)|f(\cdot,t)|^2dV'\\
&=&\int_{t\in\Pi(U)}dV_t\, \lim_{\epsilon\rightarrow 0}\int_{\Omega\cap L_t}\frac{\chi_{\epsilon}(\cdot,t)}{\epsilon}|f(\cdot,t)|^2dV'\\
&=&\int_{t\in\Pi(U)}dV_t \int_{\partial(\Omega\cap L_t)}|f(\cdot,t)|^2\frac{dS_t}{|d\rho_t|}.
\end{eqnarray*} 
Since $\Omega$ satisfies HTTP at $z^0$, it follows that there exists a constant $C>0$ such that
$$
C^{-1}\leq |d\rho|\  \ \ \text{and}\  \ \  |d\rho_t|\leq C.
$$
Hence we have
$$
\int_{K}|f|^2 dS\approx\int_{t\in\Pi(U)}dV_t \int_{\partial(\Omega\cap L_t)}|f(\cdot,t)|^2dS_t.
$$
\end{proof}
%

\subsection{$\bar\partial_b$-operator and $L^2$ existence theorem}   
Let $\Omega\subset {\mathbb C}^n$ be a bounded domain with smooth boundary and $\rho$ a defining function of $\Omega$. Following \cite{ChenShaw}, p. 166, we define the tangential Cauchy-Riemann operator $\bar{\partial}_b$ as follows. Let $\Lambda^{p,q}({\mathbb C}^n)$ denote the set of smooth $(p,q)$ forms on ${\mathbb C}^n$, and $I^{p,q}$ the ideal defined in some neighborhood $U$ of $\partial \Omega$ such that each element of the fiber $I^{p,q}_z$, $z\in U$, can be expressed as follows
   $$
   \rho u_1+\bar{\partial}\rho\wedge u_2
   $$
   for suitable $u_1\in \Lambda^{p,q}({\mathbb C}^n),\, u_2\in \Lambda^{p,q-1}({\mathbb C}^n)$. Let $\Lambda^{p,q}(\partial \Omega)$ be the orthogonal complement of $I^{p,q}|_{\partial \Omega}$ in $\Lambda^{p,q}({\mathbb C}^n)|_{\partial \Omega}$ and ${\mathcal E}^{p,q}(\partial \Omega)$ the space of smooth sections of $\Lambda^{p,q}(\partial \Omega)$ over $\partial \Omega$. Note that ${\mathcal E}^{p,q}(\partial \Omega)$ is essentially the space of $(p,q)$ forms on $\partial \Omega$. Let $\tau:\Lambda^{p,q}({\mathbb C}^n)\rightarrow \Lambda^{p,q}(\partial \Omega)$ denote the map defined by first restricting a $(p,q)$ form $u$ in ${\mathbb C}^n$ to $\partial \Omega$, then projecting the restriction to $\Lambda^{p,q}(\partial \Omega)$ when $q>0$. For $q=0$, $\tau$ is just the standard restriction map. For any $u\in {\mathcal E}^{p,q}(\partial \Omega)$, we pick first $\tilde{u}\in \Lambda^{p,q}({\mathbb C}^n)$ satisfying $\tau \tilde{u}=u$, then define $\bar{\partial}_b u$ to be $\tau \bar{\partial} \tilde{u}$ in ${\mathcal E}^{p,q+1}(\partial \Omega)$. It is easy to see that the definition of $\bar{\partial}_b$ is independent of the choice of $\tilde{u}$.

    Let $L^2(\partial \Omega)$ denote the space of functions on $\partial \Omega$ which are square integrable with respect to the induced metric from ${\mathbb C}^n$ to $\partial \Omega$, and $L^{p,q}_{(2)}(\partial \Omega)$ the space of $(p,q)$ forms on $\partial \Omega$ with $L^2$ coefficients. The operator $\bar{\partial}_b:{\mathcal E}^{p,q}(\partial \Omega)\rightarrow {\mathcal E}^{p,q+1}(\partial \Omega)$ may be extended to a linear, closed, densely defined operator
    $$
    \bar{\partial}_b: L^{p,q}_{(2)}(\partial \Omega)\rightarrow  L^{p,q+1}_{(2)}(\partial \Omega).
    $$
     A function $f\in L^2(\partial \Omega)$ is called a Cauchy-Riemann (C-R) function if $\bar{\partial}_b f=0$ holds in the sense of distributions.

     Now recall the following basic $L^2$ existence theorem of the $\bar{\partial}_b-$operator due to Shaw:

   \begin{theorem}[cf. \cite{Shaw85}, Theorem 9.3.1, Lemma 9.3.7]\label{th:shawtheorem} Let $\Omega\subset {\mathbb C}^n$ be a bounded pseudoconvex domain with smooth boundary. For any $v\in L^{p,n-1}_{(2)}(\partial \Omega)$ satisfying
   $$
   \int_{\partial \Omega} v\wedge \phi=0,\ \ \ \phi\in {\mathcal E}^{n-p,0}(\partial \Omega)\cap {\rm Ker\,}\bar{\partial}_b,
   $$
   there exists $u\in L^{p,n-2}_{(2)}(\partial \Omega)$ such that $\bar{\partial}_b u=v$ on $\partial \Omega$ and $\|u\|_{L^2(\partial \Omega)}\le {\rm const.}\|v\|_{L^2(\partial \Omega)}$. Suppose furthermore that $v$ is smooth, then $u$ is also smooth.
   \end{theorem}
\section{Proof of Theorem \ref{th: localization}}

Step 1. Let us first show that 
\begin{eqnarray*}
	\varliminf_{w\rightarrow 0} \frac{S_{\Omega'}(w)}{ S_\Omega(w)} \geq 1.
\end{eqnarray*} 
Let  $B(z, r) $ be the ball in $\C^n$ with center at $z$ and radius $r$. By assumption we may find some $r>0$ with $B(0,r)\cap\Omega\subset\Omega^{\prime}$ and $h \in \mathcal{O} (\Omega \cap B(0,r)) \cap C( \overline {\Omega \cap B(0,r) })$, such that  $h(0) =1$ and 
$$
\sup _{  \Omega \cap (B(0,r) \setminus \overline{B(0, \frac{r}{2})})} |h| <1.
$$
Let $\chi : \C^n \rightarrow [0,1]$ be a smooth function satisfying 
\begin{eqnarray}
\chi \mid _{B(0,\frac r2)} \equiv 1 , \ \ \  \chi \mid _{{B(0,r)} ^c} \equiv 0.
\end{eqnarray} 
Fix a point $w\in\Omega$ which is sufficiently close to $0$. Put 
$$
f(z) = \frac {S_{\Omega} (z,w)}{ \sqrt{S_\Omega (w)}}.
$$
For any positive integer $m$, $\chi \cdot f\cdot h^m$ and $\overline{\partial} \chi \cdot f\cdot h^m$ are smooth on $\Omega'$. We set $v_m = \tau (\overline{\partial} \chi \cdot f\cdot h^m)$. Since $\overline{\partial} (\overline{\partial} \chi \cdot f\cdot h^m) = 0$ holds in $\Omega'$, we infer from Theorem 9.2.5 of \cite{ChenShaw} that
$$
\int_{\partial \Omega'} v_m\wedge \phi=0,\ \ \ \forall\,\phi\in {\mathcal E}^{2,0}(\partial{\Omega'})\cap {\rm Ker\,}\bar{\partial}_b .
$$
By virtue of Theorem \ref{th:shawtheorem}, there exists a smooth solution of the equation $\bar{\partial}_b u_m=v_m$ on $\partial \Omega'$ such that
$$
\int_{\partial \Omega'}|u_m|^2 dS\le C\int_{\partial \Omega'} |v_m|^2dS \leq C\theta ^{2m} \int _{\partial \Omega }|f|^2 dS= C\theta ^{2m},
$$
where $C$ denotes a generic constant depending only on $\Omega$ and 
$$\theta : = \sup _{  \partial\Omega \cap (B(0,r) \setminus \overline{B(0, \frac{r}{2})})} |h|< 1.
$$
  Put $F_m = \chi \cdot f\cdot h^m - u _m$. Then $$\overline{\partial}_b F_m = \overline{\partial}_b(\chi \cdot f\cdot h^m ) - \overline{\partial}_b u_m = \tau (\overline{\partial} \chi \cdot f\cdot h^m) -\bar{\partial}_b u_m =0,$$
i.e., $F_m$ is a $L^{2}$ C-R function on $\partial \Omega'$, which admits a holomorphic extension to $\Omega'$ (still denoted by the same symbol), in view of the Bochner-Hartogs extension theorem. Note that
\begin{eqnarray} \label{eq:estimate1}
	\int _{\partial \Omega'} |F_m - fh^m|^2 dS & \leq & 2 \int_{\partial \Omega'} |1-\chi|^2 |f|^2 |h|^{2m} dS + 2 \int_{\partial \Omega'}  |u_m|^2dS \nonumber \\
	& \leq & C \theta ^{2m}  \int_{\partial \Omega'}  |f|^2 dS + C \theta ^{2m} \ \ \ \  \text {(by  \ \ Lemma \ref{le:0})} \nonumber \\
		& \leq & C \theta ^{2m} .
\end{eqnarray}
Moreover, we have $$ \int_{\partial \Omega'}  |F_m|^2 dS \leq (1+C\theta ^m) ^2.$$ It follows from the definition of the Szeg$\rm{\ddot{o}}$ kernel and (\ref{eq:estimate1})  that 
$$
|F_m(z) -f(z)h^m(z)|^2 \leq C \theta ^{2m} S_{\Omega'} (z), \ \ \ \forall\, z \in \Omega'.
$$
Hence 
$$
|F_m(z)| \geq |f(z)| |h(z)|^m - C \theta ^m \sqrt{S_{\Omega'}(z)}, \ \ \forall  \, z\in \Omega',
$$
so that
\begin{eqnarray}\label{eq:szo}
	\sqrt{S_{\Omega'} (w)} & \geq & \frac { |F_m(w)|}{\left ( \int _{\partial \Omega'} |F_m|^2dS \right)^{1/2}} \nonumber \\
	& \geq &  \frac{	|f(w)||h(w)|^m - C\theta^m \sqrt{S_{\Omega'} (w)} }{1+ C \theta^m},
\end{eqnarray}
that is,
$$
\sqrt{S_{\Omega'}(w)} \geq \frac{|h(w)|^m}{1+2C \theta^m} \sqrt{S_{\Omega} (w)}.
$$
 Hence
 $$
 \varliminf _{w\rightarrow 0} \frac{S_{\Omega'}(w)}{ S_{\Omega}(w)} \geq \frac{1}{(1+2C\theta^m)^2} \longrightarrow 1 \ \ \ (m \longrightarrow \infty).
 $$
 Step 2. It remains to prove 
 \begin{eqnarray*}
 	\varliminf_{w\rightarrow 0} \frac{S_{\Omega}(w)}{ S_\Omega'(w)} \geq 1.
 \end{eqnarray*} 
Put $$f(z)= \frac{S_{\Omega'}(z,w)}{\sqrt{S_{\Omega'}(w)}}.$$
For any positive integer $m$, we define $v_m = \tau (\overline{\partial} \chi \cdot f\cdot h^m)$. Similar as Step 1, we get a solution $u_m$ of the equation $\overline{\partial}_b u_m =v_m$ on $\partial \Omega$  satisfying
 \begin{eqnarray}\label{ineq;um}
 	\int_{\partial \Omega} |u_m|^2dS \leq C \int_{\partial \Omega}|v_m|^2 dS \leq C \theta^{2m}  \int_{\partial \Omega'}|f|^2 dS = C \theta^{2m}.
 \end{eqnarray}
Put $F_m = \chi \cdot f\cdot h^m- u _m$. We may regard $F_m$ as the boundary value of some function in $H^2(\Omega)$, which satisfies
 $$
 \int_{\partial \Omega}|F_m|^2 dS \leq (1+C\theta^m)^2.
 $$ 
 On the other hand, we have
 \begin{eqnarray} \label{ineq:1+2}
 	\int_{\partial \Omega'}|F_m -f \cdot h^m|^2dS & \leq & 2 	\int_{\partial \Omega'}|(1-\chi)f \cdot h^m|^2dS +2	\int_{\partial \Omega'}|u_m |^2dS \nonumber \\
 	& =: &  2\rm{\uppercase\expandafter{\romannumeral1}}_1 +2\rm{\uppercase\expandafter{\romannumeral1}} _2.
 \end{eqnarray}
 Note that
  \begin{eqnarray} \label{eq:I1}
\rm{\uppercase\expandafter{\romannumeral1}}_1 & \leq & C\theta^{2m} 	\int_{\partial \Omega'}|f|^2dS\leq C \theta^{2m},
 \end{eqnarray}
 and
 \begin{eqnarray}\label{eq:I2}
\rm{\uppercase\expandafter{\romannumeral1}} _2 & = & 	\int_{\partial \Omega'}|u_m|^2dS  \nonumber \\
 & = & 	\int_{\partial \Omega'\setminus \partial \Omega }|u_m |^2dS + \int_{\partial \Omega' \cap \partial \Omega }|u_m |^2dS \nonumber \\
 &  =: & \rm{\uppercase\expandafter{\romannumeral2}} _1+ \rm{\uppercase\expandafter{\romannumeral2}} _2,
 \end{eqnarray}
while
\begin{eqnarray}\label{eq:II2}
 \text{$\rm{\uppercase\expandafter{\romannumeral2}} _2$} \leq\int_{\partial\Omega}|u_m|^2 dS \leq  C \theta^{2m}.
\end{eqnarray}
Set $L_t:=\{z=(z',z_n)\in\C^n,z_n=t\}$ for $t\in\C$ and $L(R):=\bigcup_{t\in\Delta(0,R)} L_t$ for $R>0$, where $\Delta(t,R)$ denotes the disc with center $t$ and radius $R$.
By the definition of $\rho'$, there exists $R>0$ such that
$$
|d\rho'(z)|>0, \ \ \ \ \forall\, z\in L(R)\cap\Omega'.
$$
As $\Omega'$ satisfies HTTP at $0$, it follows that there exists $0<R_1<R$ such that $\Omega'_t:=\Omega'\cap L_t$, $t\in\Delta(0,R_1)$, is a smooth domain whenever it is non-empty, and we have
$$
|d_t\rho'(z',t)|>C(r)>0,\ \ \ \ \ \forall\, z\in (L(R_1)\setminus \overline{L(r)})\cap\partial\Omega',
$$
where $r$ is given above.
Then by the continuity of $|d_t\rho'(z',t)|$ there exist $\eta(r)>0$ small enough such that 
$$
|d_t\rho'(z',t)|>0,\ \ \ \ \ \forall\, z\in (L(R_1)\setminus\overline{L(r)})\cap\{\rho'>-\eta(r)\}=:U.
$$
We choose $\kappa\in C_0^\infty(\C)$ such that $\kappa\equiv 0$ in $\Delta(0,R_1)^c$ and $\kappa\equiv 1$ in $\Delta(0,r)$. It is easy to see that there exists $\epsilon_0>0$ such that
$$
|d\rho'(z)|>\epsilon_0|d\kappa(z_n)|, \ \ \ \ \ \forall\, z\in L(R_1)\cap\Omega.
$$
Set $\Omega'':=\{\rho'_{\epsilon_0}:=\epsilon_0\kappa(z_n)+\rho'(z)<0\}\subset\Omega'$ and choose $r$ small enough such that
$$
\rho'(z)>-\epsilon_0,\ \ \ \ \ \forall\, z\in\Omega\cap L(r),
$$  
which implies that $B(0,r)\cap\Omega''$ is empty.
Moreover $\Omega''$ is a smooth domain since $|d\rho'(z)|>\epsilon_0|d\kappa(z_n)|$ for all $z\in L(R_1)\cap\Omega$ and $\kappa(z_n)\equiv 0$ for all $z\in\Omega'\cap L(R_1)^c$.

Note that
$$
\int_{\partial\Omega''\cap\Omega'}|u_m|^2 dS = \int_{\partial\Omega''\cap\Omega'}|F_m|^2 dS=\int_{\partial\Omega''\cap U}|F_m|^2 dS+\int_{\partial\Omega''\cap \Omega'\setminus U}|F_m|^2 dS=:\rm{\uppercase\expandafter{\romannumeral3}} _1+\rm{\uppercase\expandafter{\romannumeral3}} _2.
$$
First of all, we have
\begin{eqnarray*}
\rm{\uppercase\expandafter{\romannumeral3}} _1&\leq&C\int_{t\in \Pi(\partial\Omega''\cap U)}dV_t\int_{\Pi^{-1}(t)\cap (\partial\Omega''\cap U)}|F_m(z',t)|^2d S'\ \ \ \ \text{(by Lemma \ref{le:3})}\\
&\leq&C\int_{t\in \Pi(\partial\Omega''\cap U)}dV_t\int_{\Pi^{-1}(t)\cap\partial\Omega}|F_m(z',t)|^2d S'\ \ \ \ \text{(by Lemma \ref{le:1})}\\
&\leq& C\int_{\partial\Omega''\cap\partial\Omega}|F_m|^2 dS\ \ \ \ \text{(by Lemma \ref{le:3})}\\
&=& C\int_{\partial\Omega''\cap\partial\Omega}|u_m|^2 dS\\
&\leq& C\theta^{2m},
\end{eqnarray*}
where $\Pi:\C^n\rightarrow\C$ denotes projection $(z',z_n)\mapsto z_n$.
Next, since $\partial\Omega''\cap\Omega'\setminus U$ is compact in $(L(R_1)\setminus\overline{L(r)})\cap\Omega'$, it follows from Bergman's inequality that $\forall\, z\in\partial\Omega''\cap \Omega'\setminus U$, 
\begin{eqnarray*}
|F_m(z)|^2 &\leq& C \int_{\Omega'\cap (L(R_1)\setminus\overline{L(r)})}|F_m|^2dV\\
&=&C\int_{t\in \Delta(0,R_1)\setminus \overline{\Delta(0,r)}}dV_t\int_{\Omega_t}|F_m(z',t)|^2dV'\\
&\leq&C\int_{t\in \Delta(0,R_1)\setminus \overline{\Delta(0,r)}}dV_t\int_{\partial\Omega_t}|F_m(z',t)|^2dS'\ \ \ \ \ \text{(by Lemma \ref{le:2})}\\
&\leq&C\int_{\partial\Omega\cap (L(R_1)\setminus \overline{L(r)})}|F_m|^2dS\ \ \ \ \ \text{(by Lemma \ref{le:3})}\\
&=& C\int_{\partial\Omega''\cap\partial\Omega}|u_m|^2 dS\\
&\leq& C\theta^{2m},
\end{eqnarray*}
which implies that $\rm{\uppercase\expandafter{\romannumeral3}} _2$ $\leq C\theta^{2m}$. Finally, we define $\widetilde\Omega:=\Omega\setminus\overline{(\Omega'\setminus\Omega'')}$. Clearly, it is a smooth domain and we have 
\begin{eqnarray}\label{eq:II1}
	\rm{\uppercase\expandafter{\romannumeral2}} _1 & = &	\int_{\partial \Omega'\setminus \partial \Omega }|u_m |^2dS = \int_{\partial \Omega'\setminus \partial \Omega }|F_m |^2dS \nonumber \\
     & \leq &	\int_{\partial \Omega''}|F_m |^2dS \nonumber \\
& \leq &	\int_{\partial \widetilde\Omega}|F_m |^2dS\ \ \ \text{(by Lemma \ \ref{le:0})}\nonumber \\
 &=&  \int_{\partial\widetilde\Omega\cap\Omega'}|F_m |^2dS +	\int_{\partial\Omega''\cap\partial\Omega}|F_m |^2dS \nonumber \\
 &=&  \int_{\partial\Omega''\cap\Omega'}|u_m |^2dS +	\int_{\partial\Omega''\cap\partial\Omega}|u_m |^2dS \nonumber \\
&\leq & C\theta^{2m}. 
\end{eqnarray} 

Combining the formula (\ref{ineq:1+2}), (\ref{eq:I1}), (\ref{eq:I2}), (\ref{eq:II2}), (\ref{eq:II1}) we obtain
 \begin{eqnarray*}
\int_{\partial \Omega'} |F_m - f \cdot h^m |^2 dS \leq C \theta^{2m}.
 \end{eqnarray*}
Again we have
\begin{eqnarray*}
	|F_m(z)-f(z)\cdot h(z)^m| \leq C \theta^m \sqrt{S_{\Omega'} (z)}, \ \ \ \forall\, z \in \Omega.
\end{eqnarray*}
Hence 
\begin{eqnarray}\label{eq:szego1}
	|F_m(z)| \geq |f(z)| \cdot |h(z)|^m - C\theta ^m \cdot \sqrt{S_{\Omega'}(z)}, \ \ \ \forall\, z \in \Omega,
\end{eqnarray}
so that

\begin{eqnarray*}
	\sqrt{S_{\Omega} (w)} & \geq & \frac { |F_m(w)|}{\left ( \int _{\partial \Omega} |F_m|^2dS \right)^{1/2}} \nonumber \\
	& \geq &  \frac{	\sqrt{S_{\Omega'} (w)} ( |h(w)|^m - C\theta^m) }{1+ C \theta^m}.
\end{eqnarray*}
Thus 
\begin{eqnarray*}
	\varliminf_{w\rightarrow 0} \frac{S_{\Omega}(w)}{ S_\Omega'(w)}  \geq \left ( \frac{ 1- C\theta^m}{1+C\theta^m} \right)^2  \rightarrow 1\ \ \ \ \ \ (m \rightarrow \infty).
\end{eqnarray*}



\section{Proof of Theorem \ref{th:exhaustszego}}
First of all, let us prove a lemma as follows
\begin{lemma}\label{szego invarance}
Let $ w:\Omega_1 \rightarrow \Omega_2$ be a biholomorphic mapping between two bounded domains with smooth boundary in $\C^n$ such that $w$ can be extended to a $C^\infty$ homeomorphism between $\overline\Omega_1$ and  $\overline\Omega_2$. If we denote  $\rho_2$ as a definition function of $\Omega_2$, then for $w=w(z)$, 
\begin{eqnarray}
d S_{w} = \lambda(z) dS_{z} \label{s1};\\
\frac{S_{\Omega_1}(z)}{\max_{\partial\Omega_1}\lambda}\leq S_{\Omega_2}(w)\leq\frac{S_{\Omega_1}(z)}{\min_{\partial\Omega_1}\lambda},\label{s2}
\end{eqnarray}
where 
$$\lambda(z):=\frac{|d\rho_1|(z)}{|d\rho_2|(w)}\cdot |J_w(z)|^2,$$
 $J_w$ is the complex Jacobian of $w$ and $\rho_1=\rho_2\circ w$ as a definition function of $\Omega_1$.
\end{lemma} 

\begin{proof}
The formula $(\ref{s1})$  follows from the fact that
\begin{eqnarray*}
dS_w&=&\frac{d\rho_2(w)}{|d\rho_2|(w)} \;\tikz{\draw[line width=0.5pt](0, 0)--(8pt, 0)--(8pt, 8pt);}\;
 \ dV_w=\frac{d\rho_1(z)}{|d\rho_2|(w)} \suob \ |J_w|^2 dV_z\\
&=&\frac{|d\rho_1|(z)}{|d\rho_2|(w)} \cdot\ |J_w(z)|^2 dS_z.
\end{eqnarray*}
On the other hand, 
for any $f\in H^2(\Omega_2)$, one have  
$$
\int_{\partial\Omega_2}|f(w)|^2dS_w=\int_{\partial\Omega_1     }|f(w(z))|^2\lambda(z)dS_z.
$$
Then we have
\begin{eqnarray}
S_{\Omega_2}(w)&=&\sup_{f\in H^2(\Omega_2)}\frac{|f(w)|^2}{\|f\|^2_{L^2(\partial\Omega_2)}}=\sup_{f\circ w \in H^2(\Omega_1)}\frac{|f\circ w(z)|^2}{\int_{\partial\Omega_1}|f\circ w(z)|^2\lambda(z)dS_z}\nonumber\\
&\leq&\frac{1}{\min_{z\in\partial\Omega_1}\lambda(z)}\sup_{f\circ w \in H^2(\Omega_1)}\frac{|f\circ w(z)|^2}{\int_{\partial\Omega_1}|f\circ w(z)|^2dS_z}\nonumber\\
&\leq&\frac{S_{\Omega_1}(z)}{\min_{z\in\partial\Omega_1}\lambda(z)}.
\end{eqnarray}
Consider $z=z(w)$ instead of $w=w(z)$, we obtain
$$
S_{\Omega_1}(z)\leq\frac{S_{\Omega_2}(w)}{\min_{z\in\partial\Omega_2}\lambda^{-1}(z(w))}=\max_{z\in\partial\Omega_1}\lambda(z)\cdot S_{\Omega_2}(w),
$$
which completes the proof.
\end{proof}
\begin{proof}[Proof of Theorem \ref{th:exhaustszego}]
We may assume that $z^0=0$. Consider the Taylor expansion series of the definition function $\rho$ of $\Omega$ at $0$
  	$$
  	\rho(z)=2{\rm Re\,} \left (  \sum_{j=1}^2 \frac{\partial \rho}{\partial z_j}(0)z_j+ \frac{1}{2} \sum_{j,k=1}^2 \frac{\partial^2 \rho}{\partial z_j\partial z_k}(0)z_jz_k \right ) + \sum_{j,k=1}^2 \frac{\partial^2 \rho}{\partial z_j\partial \bar{z}_k}(0)z_j\bar{z}_k+O(|z|^3).
  	$$
  	Furthermore, since $d\rho \neq 0$ at $0$, we choose a global coordinate transformation as follows
  	$$
  	w_1 = 2  \sum_{j=1}^2 \frac{\partial \rho}{\partial z_j}(0)z_j+ \sum_{j,k=1}^2 \frac{\partial^2 \rho}{\partial z_j\partial z_k}(0)z_jz_k,
  	$$
  	$$
  	w_2 = z_2,
  	$$
  	then
  	\begin{eqnarray}
  	\rho(w) & = &{\rm Re\,} w_1 +\frac{\partial^2 \rho}{\partial w_2\partial \bar{w}_2}(0)|w_2|^2+O(|w_1||w_2|+|w_1|^2+|w_2|^3)\nonumber\\
  	& = & x_1+ \frac{\partial^2 \rho}{\partial w_2\partial \bar{w}_2}(0)|w_2|^2 + O(|w_1||w_2|+|w_1|^2+|w_2|^3)\label{4.1}.
  	\end{eqnarray}

By Lemma \ref{szego invarance}, it suffices to estimate $S_{\Omega}(w)$. For the sake of simplicity, we still use $z$ instead of $w$.
We claim that there is a constant $c>0$ such that
   \begin{equation}\label{4.2}
   \left\{(z_1,z_2):|z_1-\varepsilon|<\varepsilon/2,|z_2|<r_\varepsilon:=c\varepsilon^{1/3}\right\}\subset {\mathbb C}^2\backslash \Omega
     \end{equation}
   for all sufficiently small $\varepsilon>0$. To see this, we first fix  a point $z$ in the bidisc for a moment. Clearly, we have $x_1>\varepsilon/2$, $|z_1|<\frac32 \varepsilon$ and $|z_2|<r_\varepsilon$, so that
    $$
    \rho(z)\ge 2x_1+\frac{\partial^2 \rho}{\partial z_2\partial \bar{z}_2}(0)|z_2|^2-C(|z_1||z_2|+|z_1|^2+|z_2|^3)\ge \varepsilon-C \left(\frac32 \varepsilon\cdot r_\varepsilon+\frac94 \varepsilon^2+r_\varepsilon^3\right)>0
    $$
    for all sufficiently small $\varepsilon>0$, provided $c$ is sufficiently small, since  $\Omega$ is pseudoconvex at $0$ (so that $\frac{\partial^2 \rho}{\partial z_2\partial \bar{z}_2}(0)\geq 0$). Take $\chi\in C^\infty(\R)$ such that $\chi_{(0,1/2)}\equiv 1$ and $\chi_{(1,+\infty)}\equiv 0.$
Put
     $$
     f_\varepsilon(z_1)=(z_1-\varepsilon)^{-1},\ \ \
    \tilde{f}_\varepsilon(z_1,z_2)=\chi(|z_2|/r_\varepsilon)f_\varepsilon(z_1).
  $$
  Clearly, $\tilde{f}_\varepsilon$ and $\frac1{z_2}\bar{\partial}\tilde{f}_\varepsilon$ are smooth on $\overline{\Omega}$.  Put
   $$
   v_\varepsilon=\tau \left(\frac1{z_2}\bar{\partial}\tilde f_\varepsilon\right).
   $$
   Since $\bar{\partial}\left(\frac1{z_2}\bar{\partial}\tilde f_\varepsilon\right)=0$ holds in $\Omega$, so we infer from Theorem 9.2.5 of \cite{ChenShaw} that
   $$
   \int_{\partial \Omega} v_\varepsilon\wedge \phi=0,\ \ \ \forall\,\phi\in {\mathcal E}^{2,0}(\partial{\Omega})\cap {\rm Ker\,}\bar{\partial}_b.
   $$
   By virtue of Theorem \ref{th:shawtheorem}, there exists a smooth solution of the equation $\bar{\partial}_b u=v_\varepsilon$ on $\partial \Omega$ such that
   $$
   \int_{\partial \Omega}|u|^2 dS\le C\int_{\partial \Omega} |v_\varepsilon|^2dS.
   $$
  By (\ref{4.1}), we may choose a constant $r_0>0$ such that for each $z\in B(0,r_0)\cap \partial \Omega\cap {\rm supp\,}v_\varepsilon$,
   $$|z_1-\varepsilon|^2\gtrsim |y_1|^2+\varepsilon^2,$$
since $|z_1-\varepsilon|\geq \frac{\varepsilon}{2}$ on ${\rm supp\,} v_\epsilon$.
 Thus
  \begin{eqnarray*}
  \int_{B(0,r_0)\cap \partial\Omega}|v_\varepsilon|^2dS
  & \le & \frac{C}{r_\varepsilon^{2}} \int_{\{|y_1|<r_0\}}(\varepsilon^2+|y_1|^2)^{-1}dy_1\int_{\{r_\varepsilon/2<|z_2|<r_\varepsilon\}}|z_2|^{-2}d\lambda_{z_2}
    \le  \frac{C}{\varepsilon r_\varepsilon^{2}}.
  \end{eqnarray*}
  We also have
  $$
  \int_{\partial \Omega-B(0,r_0)}|v_\varepsilon|^2dS\le \frac{C}{r_\varepsilon^{2}}.
  $$
  It follows that $\int_{\partial \Omega}|u|^2dS \le C(\varepsilon r_\varepsilon^2)^{-1}$.
   Put $F=\tilde{f}_\varepsilon-z_2 u$. Clearly, we have
   $$
  \int_{\partial \Omega} |F|^2dS\le 2 \int_{\partial \Omega} |f_\varepsilon|^2dS+2\int_{\partial \Omega} |u|^2dS\le C (\varepsilon r_\varepsilon^2)^{-1},
  $$
  and
   $$
   \bar{\partial}_b F=\bar{\partial}_b \tilde{f}_\varepsilon-z_2 \bar{\partial}_b u=\tau \bar{\partial}\tilde{f}_\varepsilon-z_2 \bar{\partial}_b u=z_2(v_\varepsilon-\bar{\partial}_b u)=0,
   $$
   i.e., $F$ is a smooth C-R function on $\partial \Omega$, which has a extension $\widetilde{F}\in {\mathcal O}(\Omega)\cap C^\infty(\overline{\Omega})$, in view of the Bochner-Hartogs  extension theorem. Since $\widetilde{F}(z_1,0)$ and $\tilde{f}_\varepsilon(z_1,0)$ are holomorphic on $\Omega\cap \{z:|z_1|<r_0,z_2=0\}$ and coincide on $\partial\Omega\cap \{z:|z_1|<r_0,z_2=0\}$, it follows from the uniqueness theorem of F. and M. Riesz that $\widetilde{F}(z_1,0)=\tilde{f}_\varepsilon(z_1,0)=f_\varepsilon(z_1)$ for $|z_1|<r_0$.
  In particular, $\widetilde{F}((-\varepsilon,0))=(2\varepsilon)^{-1}$ for any $\varepsilon$ small enough.

  Thus
  $$
  S_\Omega((-\varepsilon,0))\ge \frac{|\widetilde{F}((-\varepsilon,0))|^2}{\|F\|^2_{L^2(\partial \Omega)}}\ge C{r_\varepsilon^2}/{\varepsilon} \ge C\varepsilon^{-1/3}\ge C|z|^{-1/3},
  $$
and the non-tagential case is similar.

\end{proof} 

\textbf{Acknowledgements.} The authors are grateful to Professor Bo-Yong Chen for introducing us this topic and his consistent guideance. The authors also give thanks to Professor John Erik Fornaes, Yuanpu Xiong for their valuable comments.\\

\end{document}